\documentclass{article}
\usepackage{geometry}
\usepackage{graphicx} 
\usepackage{tikz}
\usepackage{url}
\usepackage{amsmath,amssymb,amsthm}
\usepackage{stmaryrd}
\usepackage{mathtools}
\usepackage{algorithmic}
\usepackage[utf8]{inputenc}
\usepackage{color}
\usepackage{mathrsfs}
\usepackage{eucal}
\usepackage{mathtools}
\usepackage{nameref}        
\usepackage{algorithmic}
\usepackage{algorithm}
\usepackage{pgfplots}
\pgfplotsset{compat=1.16}

\usepackage{url}
\usepackage{tikz,pgfplots}
\usetikzlibrary{positioning}
\pgfplotsset{compat=1.14} 
\tikzstyle{vertex} = [fill,shape=circle,node distance=30pt]
\tikzstyle{edge} = [fill,opacity=.6,fill opacity=.5,line cap=round, line join=round, line width=10pt]
\tikzstyle{elabel} =  [fill,shape=circle,node distance=30pt,fill opacity=.9]
\definecolor{mygray}{gray}{0.95}
\pgfdeclarelayer{background}
\pgfsetlayers{background,main}
\usepackage[most]{tcolorbox}
\definecolor{mypurple}{rgb}{0.59, 0.44, 0.84}
\usepackage{tikz}

\usepackage{array, makecell} 

\usepackage{afterpage}

\newtheorem{thm}{Theorem}

\newtheorem{prop}{Proposition}
\newtheorem{lemma}{Lemma}

\newtheorem{defn}{Definition}

\usepackage{cleveref}
\crefname{thm}{Theorem}{Theorems}
\Crefname{thm}{Theorem}{Theorems}
\crefname{defn}{Definition}{Definitions}
\Crefname{defn}{Definition}{Definitions}
\crefname{prop}{Proposition}{Propositions}
\Crefname{prop}{Proposition}{Propositions}
\crefname{corol}{Corollary}{Corollaries}
\Crefname{corol}{Corollary}{Corollaries}
\crefname{lemma}{Lemma}{Lemmas}
\Crefname{lemma}{Lemma}{Lemmas}

\newcommand{\R}{\mathbb{R}}

\newcommand{\xv}{\mathbf{x}}
\newcommand{\uv}{\mathbf{u}}
\newcommand{\bv}{\mathbf{b}}

\newcommand{\Av}{\mathbf{A}}
\newcommand{\Cv}{\mathbf{C}}
\newcommand{\Bv}{\mathbf{B}}

\newcommand{\tA}{\textsf{A}}

\newcommand{\fv}{\mathbf{f}}
\newcommand{\ev}{\mathbf{e}}

\newcommand{\yv}{\mathbf{y}}

\newcommand{\h}{\mathcal{H}}
\newcommand{\V}{\mathcal{V}}
\newcommand{\E}{\mathcal{E}}

\newcommand{\ie}{i.e. }

\newcommand{\e}{\mathcal{E}}

\usepackage[round]{natbib}   
\title{\bf Structural Controllability of Polysystems via Directed Hypergraph Representation}
\author{Joshua Pickard\thanks{Joshua Pickard is with the Department of Computational Medicine and Bioinformatics, Medical School, University of Michigan, Ann Arbor, MI 48109
USA (e-mail: jpic@umich.edu). This work is supported by the National Institute of General Medical Sciences
under award number GM150581.}
}

\begin{document}

\maketitle

\begin{abstract}
This document explores structural controllability of polynomial dynamical systems or polysystems. We extend Lin's concept of structural controllability for linear systems, offering hypergraph-theoretic methods to rapidly assess strong controllability. Our main result establishes that a polysytem is structurally controllable when its hypergraph representation contains no hyperedge dilation and no inaccessible vertices. We propose two efficient, scalable algorithms to validate these conditions.
\end{abstract}

\section{Introduction}
In this document we study the controllability of the system
\begin{equation}\label{eq: polynomial with linear inputs}
    \dot{\xv}=\fv(\xv)+\Bv\uv
\end{equation}
where $\xv\in\R^n$ is the system state, $\fv$ is an odd, homogeneous polynomial of $\xv$, $\Bv\in\R^{n\times m}$ is the control matrix, and $\uv\in\R^m$ is the system input. Limitations in the precision to which the entries of $\fv$ and $\Bv$ are known motivate structural controllability, initially introduced for linear dynamical systems.

\begin{defn}[Structurally Controllable, \cite{lin1974structural}]
    Consider the linear dynamical system $\dot{\xv}=\Av\xv+\Bv\uv$ where $\Av\in\R^{n\times n}.$ The pair $(\Av,\Bv)$ is structurally controllable if and only if there exist a completely controllable pair $(\Hat{\Av},\Hat{\Bv})$, where $\hat{\Av}$ is any matrix with the same sparsity structure as $\Av.$
\end{defn}

There are several advantages of strutructural -- over complete or strong -- controllability. Graph theoretic tests to verify structural controllability are faster than PHB or Kalman-based tests required for complete or strong controllability (\cite{liu2011controllability}); system identification is rarely performed to infinite precision, which limits the claims we may make about the complete/strong controllability of engineered or biological systems (\cite{lee1967foundations}); while dynamics of many real systems can't be described perfectly as linear or polynomial systems, the connectedness of the network can be precisely identified. This motivated initial tests for structural controllability of a linear system with a single input where $\bv\in\R^n$ takes the place of $\Bv.$

\begin{thm}[\cite{lin1974structural}]\label{thm: lin 1974}
    The following properties are equivalent to structural controllability of a linear system defined by the pair $(\Av,\bv):$
    \begin{itemize}
        \item There is no permutation of the cooridinates to bring the pair $(\Av,\bv)$ to the form
        \begin{equation*}
            \begin{bmatrix}
                \Av_{11}&0\\\Av_{21}&\Av_{22}
            \end{bmatrix}\text{ and }\begin{bmatrix}
                0\\\bv_2
            \end{bmatrix}\text{ or }rank\big(\begin{bmatrix}
                \Av&\bv
            \end{bmatrix}\big)<n
        \end{equation*}
        \item The graph of $(\Av,\bv)$ is span by a cactus
        \item The graph of $(\Av,\bv)$ contains no inaccessable node and no dilation
    \end{itemize}
\end{thm}

Several extensions were made immediately after the work of \cite{lin1974structural} to consider multiinput control systems (\cite{shields1976structural}), strong structural controllability (\cite{mayeda1979strong}), and other similar varieties of the problem. Inspired by \cite{liu2011controllability}, there has been a recent resurgence in the study and utilization of structural controllability (\cite{mousavi2018null, gao2014target}). While there are limitations to the structural controllability (\cite{cowan2012nodal}), this avenue of work presents several advantages.

In this document, we define and propose conditions for the structural controllability of odd, homogeneous polynomial systems with linear inputs defined by a pair $(\fv,\Bv)$ or $(\tA,\Bv)$ where $\tA$ is the tensor representation of $\fv.$

\begin{thm}\label{thm: structural control of polynomials}
    A pair $(\tA,\Bv)$ is structurally controllable if there exist a strongly controllable pair $(\hat{\tA},\hat{\Bv}),$ which occurs if and only if the hypergraph of $(\tA,\Bv)$ contains no hyperedge dialation and no inaccessable vertices.
\end{thm}

We will prove \Cref{thm: structural control of polynomials} and develop a pair of algorithms that verify this condition in quadratic time throughout this document, which is organized as follows. \Cref{sec: polynomial control} provides preliminary results on polynomial control systems and their relationship to uniform, undirected hypergraphs, and then \Cref{sec: directed hypergraphs} provides several definitions for directed hypergraph structures. In \Cref{sec: main result}, \Cref{thm: structural control of polynomials} is proven using a directed hypergraph representation of polynomial systems. Finally, in \Cref{sec: algorithms}, fast hypergraph algorithms are proposed to test for structural controllability. 



\section{Polynomial Systems}\label{sec: polynomial control}

\cite{jurdjevic1985polynomial} derived conditions for the strong controllability of odd homogeneous polynomials with linear inputs.\footnote{A polynomial $p(\xv)$ is homogeneous if $p(\lambda\xv)=\lambda^dp(\xv).$ $d$ is the degree of the polynomial $p$ and the parity of $p$ is that of $d.$}

\begin{thm}[\cite{jurdjevic1985polynomial}]\label{thm: jurdjevic and kupka}
    Suppose that $\textbf{f}$ is a homogeneous polynomial system of odd degree. Consider the following system 
    \begin{equation}\label{eq: jk equation}
        \dot{\textbf{x}} = \textbf{f}(\textbf{x}) + \sum_{j=1}^m\textbf{b}_ju_j.
    \end{equation}
    Then system \ref{eq: jk equation} is strongly controllable if and only if the rank of the Lie algebra spanned by the set of vector fields $\{\textbf{f}, \textbf{b}_1,\textbf{b}_2,\dots,\textbf{b}_m\}$ is $n$ at all points of $\mathbb{R}^n$.  Moreover, the Lie algebra is of full rank at all points of $\mathbb{R}^n$ if and only if it is of full rank at the origin. 
\end{thm}

The rank of the Lie algebra can be found by evaluating the recursive Lie brackets of $\{\textbf{f},\Bv_{:1},\dots,\Bv_{m:}\}$ at the origin. The Lie bracket of two vector fields $\textbf{f}$ and $\textbf{g}$ at a point \textbf{x} is defined as
\begin{equation}\notag
    [\textbf{f},\textbf{g}]_\textbf{x} = \nabla\textbf{g}(\textbf{x})\textbf{f}(\textbf{x}) - \nabla\textbf{f}(\textbf{x})\textbf{g}(\textbf{x}),
\end{equation}  
where $\nabla$ is the gradient operator. Directly evaluating the Lie algebra generated by a polynomial vector field and set of linear inputs that act as constant vector fields is computationally intractable for large systems; however, the tensor representation of a polynomial allows for the utilization of the Singular Value Decomposition (SVD) to accelerate construction of the Lie algebra.

\subsection{Tensor Representation of Polynomials}
Any homogeneous polynomial can be represented through repeated tensor contractions.

\begin{defn}[Tensor Representation of Polynomials]
    Given a $k$-mode, $n$-dimensional tensor $\tA$ and a vector $\xv\in\R^n,$ the repeated tensor-vector multiplication or contraction
    \begin{equation}\label{eq: homogeneous polynomial}
        \tA\times_1\xv\times_2\dots\times_{k-1}\xv=\tA\xv^{k-1}
    \end{equation}
    is referred to as the homogeneous polynomial of $\tA.$\footnote{We use $\times_i$ to denote the Tucker product or contraction of $\xv$ with the $i$th mode of $\tA.$ Throughout the remainder of the article, we adopt the conventions of multilinear algebra presented in \cite{kolda2009tensor}}
\end{defn}

We are able to represent any homogeneous polynomial in the form of \cref{eq: homogeneous polynomial}, and when $\tA$ has $k$-modes or indices, the expression $\tA\xv^{k-1}$ is a $k-1$th degree, homogeneous polynomial. Then, when $k$ has even parity, $\tA\xv^{k-1}$ is an odd homogeneous polynomial and we are able to apply \cref{thm: jurdjevic and kupka}. In the remainder of this document, pairs $(\fv,\Bv)$ and $(\tA,\Bv)$ refer to odd homogeneous polynomials and are used interchangeably, and we may rewrite our equation of interest \cref{eq: polynomial with linear inputs} as
\begin{equation}\label{eq: main equation}
    \dot{\xv}=\fv(\xv)+\Bv\uv\Longleftrightarrow\dot{\xv}=\tA\xv^{k-1}+\Bv\uv.
\end{equation}

The homogeneous polynomial of $\tA\xv^{k-1}$ can be expressed equivalently through tensor matricization and Kronecker exponentiation as
\begin{equation}\label{eq: tensor unfolded polynomial}
    \tA\xv^{[k-1]}=\tA_{(k)}\times(\overbrace{\xv\otimes\dots\otimes\xv}^{\text{$k-1$ times}})=\tA_{(k)}\xv^{[k-1]},
\end{equation}
where $\tA_{(k)}$ denotes the $k$th mode unfolding of $\tA$ (\cite{kolda2009tensor, ragnarsson2012block}) and $\xv^{[i]}$ denotes Kronecker exponentiation (\cite{pickard2023kronecker}). The polynomial representation provided in \cref{eq: tensor unfolded polynomial} allows for the construction of a nonlinear controllability matrix.

\subsection{A Test for Polynomial Controllability}
A Kalman-rank-like test for controllability is developed in \cite{chen2021controllability} where \cref{eq: main equation} is strongly controllable if and only if $rank(\Cv(\tA,\Bv))=n$ where $\Cv(\tA,\Bv)$ is the nonlinear controllability matrix defined
\begin{equation}\label{eq: nonlinear controllability matrix}
    \Cv(\tA,\Bv)=\begin{bmatrix}
        \Bv&\tA_{(k)}\Bv^{k-1}
        & \dots & \tA_{(k)}(\tA_{(k)}(\dots\tA_{(k)}\Bv^{[k-1]})^{[k-1]}\dots )^{[k-1]}
    \end{bmatrix}.
\end{equation}
In practice, the authors of \cite{chen2021controllability} do not explicitly construct $\Cv(\Av,\Bv)$ but rather developed a fast SVD based algorithm to construct a reduced controllability matrix $\Cv_r(\tA,\Bv)$ that will have the same rank; see \cref{alg:1} for details.

Initially proposed for undirected (or really omnidirectional\footnote{Undirected hypergraphs in \cite{chen2021controllability} have dynamics defined so that if there are $k$ vertices in a hyperedge, then information propagates from all sets of $k-1$ vertices to the remaining $k$th vertex. Similar to how an undirected graph has a symmetric adjacency matrix and can be thought of as bidirectional, the dynamics of undirected hypergraphs flow not in bi- or two directions, but in all possible directions indicated by the vertices in the hyperedge. In this sense, undirected hyperedges are omnidirectional.}) hypergraphs, the adjacency tensor $\tA$ is supersymmetric. As a result, the published version of \cref{alg:1} proposed that all unfoldings of $\tA$ are equivalent; however, to consider a generic homogeneous polynomial where directionality is encoded in $\fv$ and asymmetries of $\tA,$ we must enforce that the unfolding $\tA_{(k)}$ is utilized in line with \cref{eq: tensor unfolded polynomial}.
\begin{algorithm}
\caption{\cite{chen2021controllability}}
\label{alg:1}
\begin{algorithmic}[1]
    \REQUIRE $(\tA,\Bv)$
    \STATE $\Av=\tA_{(k)}$
    \STATE{Set $\textbf{C}_r = \textbf{B}$ and $j=0$}\\
    \WHILE{$j < n$}
        \STATE{Compute $\textbf{L}=\textbf{A}(\textbf{C}_r\otimes \textbf{C}_r\otimes \stackrel{k-1}{\cdots} \otimes \textbf{C}_r)$}\\
        \STATE{Set $\textbf{C}_r=\begin{bmatrix} \textbf{C}_r & \textbf{L}\end{bmatrix}$}\\
        \STATE{Compute the economy-size SVD of $\textbf{C}_r$, and remove the zero singular values, i.e., $\textbf{C}_r=\textbf{U}\textbf{S}\textbf{V}^\top$ where $\textbf{S}\in\mathbb{R}^{s\times s}$, and $s$ is the rank of $\textbf{C}_r$}
        \STATE{Set $\textbf{C}_r=\textbf{U}$, and $j=j+1$}
    \ENDWHILE
    \RETURN The reduced controllability matrix $\textbf{C}_r$.
\end{algorithmic}
\end{algorithm}

\subsection{Limitations of Strong Controllability}


As it was argued by \cite{lin1974structural}, even for linear systems $(\Av,\Bv),$ the system identification problem of identifying the parameter entries of $\Av$ and $\Bv$ can be determined up to numerical or more often experimental precision, which is problematic when considering that the set of all completely controllable pairs $(\Av,\Bv)$ is open and dense.

\begin{thm}[\cite{lee1967foundations}]\label{thm: Av Bv is open and dense}
    For any pair $(\Av,\Bv)$ that is not completely controllable and any $\varepsilon\in\R,$ there exist a completely controllable pair $(\Av',\Bv')$ where $\|\Av-\Av'\|<\varepsilon$ and $\|\Bv-\Bv'\|<\varepsilon.$
\end{thm}

When we consider our limitations in determining the parameters of $(\Av,\Bv)$ in the context of \cref{thm: Av Bv is open and dense}, there will always be a completely controllable pair arbitrarily close to the system $(\Av,\Bv)$ that is identified. This problem exist equivalently for pairs $(\tA,\Bv).$

\begin{prop}\label{prop: open dense multlilinear}
    The set of odd homogeneous polynomial controllable systems defined by the pairs $(\tA,\Bv)$ is open and dense: For any pair $(\tA,\Bv)$ that is not strongly controllable and any $\varepsilon\in\R,$ there exist a strongly controllable pair $(\tA',\Bv')$ such that $\|\tA-\tA'\|<\varepsilon$ and $\|\Bv-\Bv'\|<\varepsilon$.
\end{prop}

\begin{proof}
    This proof follows a similar structure to the original proof that $(\Av,\Bv)$ is open and dense presented on page 100 in Theorem 11 of \cite{lee1967foundations} with the notable exception that we exchange the linear controllability matrix for the nonlinear controllability matrix $\Cv(\Av,\Bv).$
\end{proof}


The limitation of \cref{prop: open dense multlilinear} motivates the definition of structural controllability for odd homogeneous polynomial systems.

\begin{defn}
    An odd homogenenous polynomial represented by the pairs $(\fv,\Bv)$ or $(\tA,\Bv)$ is structurally controllable if and only if there exist a strongly controllable pair $(\Hat{\tA},\Hat{\Bv}).$
\end{defn}

\subsection{Hypergraph Preliminaries}
Hypergraphs provide a convienient framework for representing the sparsity structure $(\tA,\Bv)$ and were the initial motivation of \cref{alg:1}. A hypergraph $\h=\{\V,\E\}$ is a finite set of vertices $\V$ together with a set of hyperedges where each hyperedge $\e\in\E$ is a subset of vertices i.e. $e\subseteq\mathcal{P}(\V)\backslash\{\emptyset\},$ where $\mathcal{P}(\V)$ denotes the power set of $\V$ and $\backslash$ is the set difference. A hypergraph is $k$-uniform if all hyperedges $e$ contain $k$ vertices. The adjacency tensor is the higher order analogue of an adjacency matrix which represents the structure of a hypergraph.

\begin{defn}[\citet{cooper2012spectra}]
    Let $\h=\{\V,\E\}$ be a $k$-uniform hypergraph with $n$ nodes. The \textit{adjacency tensor} $\textsf{A}\in\mathbb{R}^{n\times n\times \dots\times n}$ of $\h$, which is a $k$-th order $n$-dimensional supersymmetric tensor, is defined as
    \begin{equation}\label{eq:98}
    \textsf{A}_{j_1j_2\dots j_k} = \begin{cases}
            \frac{1}{(k-1)!} & \text{if $\{j_1,j_2,\dots,j_k\}\in \E$}\\
            0 & \text{otherwise}
        \end{cases}.
    \end{equation}
\end{defn}

The homogeneous polynomial of the adjacency tensor can be fit to any homogeneous polynomial $\fv(\xv)$ by modifying the values of the coefficients or tensor entries in \cref{eq:98}, and conditions and fast algorithms have been developed for the observability and controllability of hypergraph dynamics with linear inputs and outputs.

\begin{defn}[\cite{chen2021controllability, pickard2023observability}]
    Given a $k$-uniform hypergraph $\h$ with $n$ nodes, the dynamics of $\h$ with control inputs can be represented by 
    \begin{equation}\label{eq: HG dynamics}
    \begin{cases}
        \dot{\xv} & = \tA\xv^{k-1} + \sum_{j=1}^m\textbf{b}_ju_j\\
        \yv& = \Cv\xv
    \end{cases}
    \end{equation}
    where, $\textsf{A}\in\mathbb{R}^{n\times n\times \dots \times n}$ is the adjacency tensor of \textsf{G}, and $\textbf{B}=\begin{bmatrix}\textbf{b}_1 & \textbf{b}_2 & \dots & \textbf{b}_m\end{bmatrix}\in\mathbb{R}^{n\times m}$ is the control matrix.
\end{defn}

We will utilize this hypergraph framework to develop conditions for structural control in \cref{sec: main result}. In the following section, we relax some of the above requirements and develop new structures to discuss directed hypergraphs.


\section{Directed Hypergraphs}\label{sec: directed hypergraphs}

In this section, we describe one method through which the hypergraph representation of an odd, homogeneous polynomial $\fv$ can be constructed. In order to apply \cref{thm: jurdjevic and kupka} or \cref{alg:1} to determine strong controllability of \cref{eq: main equation}, we require odd, homogeneous polynomial dynamics. This means that each term appearing in the right hand side of $\dot{\xv}=\fv(\xv)$ must be the product of an odd number of entries in $\xv$ (these will be the tail edges), and to accommodate tensor representation the number of entries from $\xv$ must be uniform accross all terms in $\fv(\xv)$; however, there is no restriction from \cref{thm: jurdjevic and kupka} as to how many entries of $\dot{\xv}$ that a given hypertail can appear in (entries of $\dot{\xv}$ with identical tails will be the heads).

\subsection{Definition and Tensor Representation}
To develop a condition for structural control, we assign a directed hypergraph structure to the pair $(\tA,\Bv)$. Given a polynomial $\dot{\xv}=\fv(\xv),$ the dynamics of each term $\dot{\xv}_i$ can be written as
\begin{equation*}
    \dot{\xv}_i=\sum_{XX\text{ given }\xv_i}\ \prod_{YY\text{ given }XX}\xv_j,
\end{equation*}
where $XX$ and $YY$ are conditional expressions. 
In an undirected/omnidirectional hypergraph, the conditions $XX$ and $YY$ are
\begin{equation}
    \dot{\xv}=\tA\xv^{k-1}+\Bv\uv\Longrightarrow\text{   }\dot{\xv}_i=\sum_{e\in\E|v_i\in e}\prod_{j|v_j\in e,i\neq j}\xv_j,
\end{equation}
or when the control inputs are treated separately, the system is written
\begin{equation}\label{eq: hg dynamics by directed hg structure}
    \dot{\xv}_i=\overbrace{\sum_{e\in\E|v_i\in e}\prod_{j|v_j\in e,i\neq j}\xv_j}^{\tA\xv^{k-1}}+\Bv_{i:}\uv.
\end{equation}
Already, there is a notion of directionality where ``information" or ``forcing" propagates from $\xv_j$ to $\xv_i.$

In a directed \textit{graph}, we would say that $\xv_i$ is in the \textit{edge} head and $\xv_j$ is in the \textit{edge} tail. The condition that $\xv_i$ reside in the head, modifies which hyperedges are summed over or which hyperedges influence $\xv_i$ with the $XX$ conditional, and the condition that $\xv_j$ is in the tail modifies which vertices are considered in the product via the $YY$ conditional. This can be extended to a directed hypergraph as
\begin{equation}\label{eq: directed hypergraph dynamics definition}
    \dot{\xv}=\tA\xv^{k-1}+\Bv\uv
    \Longleftrightarrow    \dot{\xv}_i=\overbrace{\sum_{e\in\E|v_i\in e_h}\prod_{j|v_j\in e_t}\xv_j}^{\tA\xv^{k-1}}+\Bv_{i:}\uv,
\end{equation}
where $e^t$ and $e^h$ denote hyperedge tails and heads respectively. From \cref{eq: directed hypergraph dynamics definition}, we can explicitly define the hyperedge heads $e^h$ and tails $e^t$ to construct a directed hypergraph representation of $(\tA,\Bv).$


\begin{defn}\label{def: HG of tA Bv}
    Given a polysystem $(\tA,\Bv),$ there the directed hypergraph $\h(\tA,\Bv)=\{\V,\E\}$ is defined:
    \begin{itemize}
        \item $\V=\{\overbrace{v_1,\dots,v_n}^{\text{system vertices}},\overbrace{v_{n+1},\dots,v_{n+m}}^{\text{control vertices}}\}$
        \item $\E=\{(e^t_1,e^h_1),\dots,(e^t_l,e^h_l)\}$ where each $e^t_i\subseteq \V$ and each $e^h_i\subset \V$
        \item Each hypertail is unique (\ie  if $e^t_i=e^t_j$ then $i=j$)
        \item Control vertices are not in the hyperedge heads (\ie  $e^h\cap\{v_{n+1},\dots,v_{n+m}\}=\emptyset$)
    \end{itemize}
\end{defn}
Several modifications are made between directed and undirected/omnidirectional hypergraphs:
\begin{itemize}
    \item Direction: each hyperedge can be explicitly split into hypertails and hyperheads indicating how a walk, flow, or control input will propagate along the hypergraph.
    \item Nonuniformity: to apply \cref{thm: jurdjevic and kupka}, only the hypertails are fixed sizes. The heads can vary in size and therefore so can the size of any hyperedge.
\end{itemize}
Rather than constructing a new adjacency tensor for $\h(\tA,\Bv),$ the set of all tensors with the sparsity $\hat{\tA}$ will be the only objects required to verify if the system is structurally controllable.


\subsection{Hypergraph Structures}
To extend \cref{thm: lin 1974}, we define accessible vertices and dilations on hypergraphs, which a few definitions; the following definitions will be utilized in \cref{sec: main result}.

In a graph, a set of vertices $S\subset\V$ is dilated if the in-neighborhood of $S,$ denoted $\Gamma(S)$ is smaller than $S$ \ie if $|S|>\Gamma(S),$ then $S$ is dilated (\cite{lin1974structural}). For hypergraphs, we define a dilation not by the size of the neighborhood of $S$ but by the number of hyperedges directed into $S$.

\begin{defn}[Hyperedge Dialation\footnote{I have not seen definitions of hypergraph dilations in the literature yet.}]
    Given a directed hypergraph $\h=\{\V,\E\}$ where each directed hyperedge is denoted $(e^t,e^h)\in\E$, a set of vertices $\{v_1,\dots,v_s\}$ is dilated if the set of heads directed at at least one vertex in  $\{v_1,\dots,v_s\}$ is smaller than $s$.
\end{defn}

This definition is convenient because the existince of such a dilation can be verified as a maximum matching problem in the star expansion of a directed hypergraph; see \cref{sec: algorithms}.

For a traditional graph or 2-uniform hypergraph, a walk is a sequence of edges $e_1,\dots,e_l$ where the edges are aligned $e_i^h=e_{i+1}^t.$ Graph walks have two distinct properties:
\begin{itemize}
    \item \textit{Immediately} before $e_i$ is added to the walk, the vertices in $e_i^t$ are visited.
    \item \textit{All} vertices in the tail $e^t_{i+1}$ are visited before $e_i$ is included in the walk.
\end{itemize}

Typically, hypergraph walks focus on the first condition, requiring that nearby vertices and hyperedges are visited sequentially. Several varieties of hypergraph walks exist (\citet{gallo1993directed,ausiello2017directed, aksoy2020hypernetwork}). 

Following the second condition, we take an alternative perspective on hypergraph walks: we impose the strict requirement that all vertices in the tail $e^t_i$ are visited at some point in the walk prior to hyperedge $e_i$ being visited. This relaxes the condition of visiting vertices in a consecutive order.

\begin{defn}[Walk \footnote{There are many definitions of hypergraph walks in the literature; however, I have not found a definition equivalent to this one yet.}]
    A hypergraph walk beginning with the sets of visited vertices $\{\{v_1\},\dots,\{v_a\}\}$ is a sequence of hyperedges $w=(e_1,\dots,e_l)$ that obeys the following conditions:
    \begin{itemize}
        \item A set of visited sets of vertices is maintained throughout the walk, beginning with $V_0=\{\{v_1\},\dots,\{v_a\}\}$
        \item Prior to $e_i$ appearing in the walk sequence, all vertices in the tail  $e_i^t$ are found in sets of visited vertices that contain no vertices that are not in the tail of $e_i$, i.e. $\exists\ h_1,\dots,h_l\in V_{i-1}$ such that $h_1\cup\dots\cup h_l=e^t_i$ and $h_j\backslash e^t_i=\emptyset\ \forall j,$ where $\backslash$ denotes the set difference.
        \item When a new hyperedge $e_i$ is added to the walk sequence $w$, vertices in the head $e_i^h$ are counted as visited and added as a new set to the list of sets of vertices that may be used to walk on new hyperedges i.e. $V_i=V_{i-1}\cup e_i^h.$
    \end{itemize}
    For a walk $w$ of length $l,$ the vertices $e^h_l$ are the final vertices of $w.$
\end{defn}

This definition is convenient because it is simple and fast to perform a breadth first search of a directed hypergraph using this type of walk. Furthermore, it allows us to determine the set of hyperheads accessible from a given set of visited vertices $V_i$ as the span of
\begin{equation*}
    \hat{\tA} \times_1 v_1 \times_2 \ldots \times_{k-1} v_{k-1}
\end{equation*}
where $v_i \in V_i$. Notably, this expression bears a helpful resemblance to \cref{eq: main equation}.

The notion of an accessible vertex for hypergraph walks parallels that of a walk on a graph.

\begin{defn}[Accessible Vertices]
    Given a set of vertices $V_0=\{\{v_1\},\dots,\{v_a\}\},$ a vertex $v$ is accessible if there exist a walk beginning at $V_0$ that visits $v.$
\end{defn}

Unique to hypergraph walks, some vertices are accessible only in groups. For instance, if $v_1$ and $v_2$ are each only in the head of one hyperedge and $v_1$ and $v_2$ are in the same hyperedge, then all walks must visit $v_1$ and $v_2$ together. We refer to a vertex $v$ as individually accessible if there exist a set of walks that will isolate $v.$

\begin{defn}[Individually Accessible Vertex]
    A vertex $v$ is individually accessible if there exist a set of walks $w_1,\dots,w_k$ such that $\{v\}$ is the result of a series of set unions and differences between the final vertices of walks $w_1,\dots,w_k$.
\end{defn}

With these definitions, we are equipped to determine what sparsity structure of $(\tA,\Bv)$ will render $\Cv(\tA,\Bv)$ full rank and to verify \cref{thm: structural control of polynomials}.

\section{Structural Controllability for Polysystems}\label{sec: main result}
In this section we return to and prove our main result of \cref{thm: structural control of polynomials}. We show that the presence of an inaccessible vertex or a hyperedge dilation is both sufficient and necessary for the pair $(\tA,\Bv)$ to be structurally uncontrollable. Our argument is based upon the rank of the nonlinear controllability matrix $\Cv(\tA,\Bv)$ where a pair $(\tA,\Bv)$ is structurally uncontrollable if there exist no pair $(\hat{\tA},\hat{\Bv})$ with a full rank controllability matrix.

\subsection{Sufficient}
Here we show that if $\h(\tA,\Bv)$ contain either a inaccessible vertex or a hyperedge dilation then there exist no pair $(\hat{\tA},\hat{\Bv})$ with a full rank controllability matrix.



\begin{prop}\label{prop: inaccessible is not SC}
    In the hypergraph $\h(\tA,\Bv),$ if there exist a vertex in $\{1,\dots,n\}$ that is not accessible from a walk beginning at $\{\{n+1\},\dots,\{n+m\}\},$ then the pair $(\tA,\Bv)$ is not structurally controllable.
\end{prop}

\begin{proof}
    If there exist a inaccessible vertex $v$, from the control vertices, then there is no walk beginning at the control vertices which will ever visit $v.$ Denote by $\mathscr{W}_1$ the set of vertices that are accessible from the control vertices such that $\mathscr{W}_1$ spans the column vectors of $\Bv.$ 
    The set of vertices accessible after the $i$th step in the walk is the span of $\mathscr{W}_i$ where 
    \begin{equation}
        \mathscr{W}_i=\mathscr{W}_{i-1}\cup\{\hat{\tA}\times_2\ev_{i_1}\times\dots\times_{k}\ev_{i_{k-1}}\text{ where }\ev_i\in \mathscr{W}_{i-1}\}.
    \end{equation}
    If $v$ is inaccessible, then there exist no $i$ such that $\mathscr{W}_i$ contains a vector whose $i$th entry is nonzero. In the context of the controllability matrix, we observe that $\Cv(\tA,\Bv)$ is the matrix whose columns are generated following a similar procedure as $\mathscr{W}_i$ for $i\geq n,$ so if there exist a inaccessible vertex $v$ such that the $v$th entry of all vectors in $\mathscr{W}$ is always 0, similarally, the $v$th row of $\Cv(\tA,\Bv)$ will be all zeros and $rank(\Cv(\tA,\Bv))<n.$ 
\end{proof}

\begin{prop}\label{prop: non individually accessible is not SC}
    In the hypergraph $\h(\tA,\Bv)$, if there exist a vertex $v_i$ that is not individually accessible, then the pair $\Cv$ is not structurally controllable.
\end{prop}

\begin{proof}
    Suppose $v_i$ is accessible but not individually accessible, and let $v_j$ be the vertex that is always visited in conjunction with $v_i$ such that $\{v_i,v_j\}$ is accessible by neither $\{v_i\}$ nor $\{v_j\}$ are individually accessible. Then, it follows that the span of $\mathscr{W}$ contains vectors with nonzero entries that are linearly dependent in the $i$th and $j$th position. 
    Then, since $\Cv(\tA,\Bv)$ is constructed such that its columns span a similar set of vectors at $\mathscr{W},$ there exist a linear dependence between the $i$th and $j$th rows of $\Cv(\tA,\Bv)$, so $rank(\Cv(\tA,\Bv))<n,$ and $(\tA,\Bv)$ is not structurally controllable.
\end{proof}

\begin{prop}\label{prop: dilaiton is not SC}
    There exist a pair of accessible vertices $\{v_1,v_2\}$ where $v_1$ and $v_2$ are not individually accessible if and only if a hyperedge dilation occurs.
\end{prop}

\begin{proof}
    ($\Longrightarrow$) Suppose $\{v_1,v_2\dots,v_k\}$ are accessible from walk $w_1,$ and assume there is no dilation containing vertices $v_1$ and $v_2.$ Then there must exist at least two hyperedges $e_1$ and $e_2$ that contain vertices $v_1$ and $v_2$ in the hyperedge heads. Since $e_1$ and $e_2$  have distinct heads from one another with respect to $v_1$ and $v_2$, their heads must contain two of the three: $\{v_1\},\{v_2\},$ or $\{v_1,v_2\}.$ However, if $\{v_1\}$ and $\{v_2\}$ are the heads than both $v_1$ and $v_2$ are immediately individually accessible (a contradiction), and if $\{v_1,v_2\}$ and $\{v_1\}$ or $\{v_2\}$ are the heads, then $v_1$ and $v_2$ are accessible through the set difference of these heads (another contradiction). These contradictions violate our assumption, so the assumption that no dilation occurs is false, and we may conclude that when vertices are accessible but not individually accessible, a hyperedge dilation exist.

    ($\Longleftarrow$) If there exist a dilation, then there must exist individually inaccessible vertices.
\end{proof}

\begin{lemma}
    In the hypergraph $\h(\tA,\Bv)$, the presence of either a nonaccessable vertex or a hyperedge dilation are a sufficient condition for the pair $(\tA,\Bv)$ to be structurally uncontrollable.
\end{lemma}

\begin{proof}
    This follows as an immediate consequence of the previous propositions. If there is a inaccessible vertex, then \cref{prop: inaccessible is not SC} says $(\tA,\Bv)$ is not structurally controllable. If there is a dilation, then \cref{prop: dilaiton is not SC} says there exist individually inaccessible vertices, which \cref{prop: non individually accessible is not SC} says implies the system is not structurally controllable.
\end{proof}

\subsection{Necessary}
The next step is to show that the presence of a hyperedge dilation or inaccessible vertex is a necessary condition for $(\tA,\Bv)$ to be structurally uncontrollable.

\begin{prop}
    Suppose $(\tA,\Bv)$ is not structurally controllable. Then $\h(\tA,\Bv)$ contains either a hyperedge dilation or  inaccessible vertex.
\end{prop}

\begin{proof}
    If $(\tA,\Bv)$ is not structurally controllable, then $rank(\Cv(\tA,\Bv))<n$ for all entries in the set $(\hat{\tA},\hat{\Bv}).$ This can occur either if $rank(\begin{bmatrix}
        \hat{\tA}_{(k)}&\hat{\Bv}
    \end{bmatrix})<n$ or if $rank(\begin{bmatrix}
        \hat{\tA}_{(k)}&\hat{\Bv}
    \end{bmatrix})=n,$ where $\begin{bmatrix}
        \hat{\tA}_{(k)}&\hat{\Bv}
    \end{bmatrix}\in\R^{n\times n(k-1)+m}.$ Beginning with the case $rank(\begin{bmatrix}
        \hat{\tA}_{(k)}&\hat{\Bv}
    \end{bmatrix})<n,$ if this matrix is generally rank deficient, then there exist at least one row containing all zeros -- corresponding to an inaccessible vertex -- or if there are at least $n(k-2)+m+1$ columns containing all zeros -- corresponding to a dilation.
    
    Consider the alternative case where $rank(\begin{bmatrix}
        \hat{\tA}_{(k)}&\hat{\Bv}
    \end{bmatrix})=n.$ Then $dim(Im(\hat{\tA}_{(k)})\cup Im(\Bv))=n$ and $rank(\Cv(\tA,\Bv))<n.$ Since $rank(\Cv)<n,$ it follows that the dimension of $Im(\hat{\tA}_{(k)})$ shrinks when restricted to the domain of the image of $\Bv.$ This indicates that part of the image of $\hat{\tA}_{(k)}$ is inaccessible from the image of $\Bv,$ which is equivalent to their being a vertex that is not individually accessible from walks originating at the vertices of $\Bv.$ Since there is a vertex that is not individually accessible, there must exist either a inaccessible vertex or a hyperedge dilation. 
\end{proof}

We can now conclude that presence of inaccessible vertices or a hyperedge dilations is both neccessary and sufficient for a pair $(\tA,\Bv)$ to be structurally uncontrollable, and we have proved \cref{thm: structural control of polynomials}.

\section{A Test for Structural Controllability}
\label{sec: algorithms}

This section presents 2 algorithms to verify the separate conditions of \cref{thm: structural control of polynomials}, namely the existence of inaccessible nodes and hyperedge dilations. Checking for dilations is equivalent to solving a maximum matching problem, which can be accomplished in $\mathcal{O}(|\V|\sum_i|e_i^h|),$ and verifying the existence of inaccessible nodes can by performing a walk beginning at the control vertices over all hyperedges, which occurs in $\mathcal{O}(|\E|).$ The complexity to run both algorithms and verify \cref{thm: structural control of polynomials} is $\mathcal{O}(|\V|\sum_i|e_i^h|+|\E|).$

\subsection{Maximum Matching to Detect Dilations}
Maximum matching was utilized by \cite{liu2011controllability} to determine structural controllability of graphs/linear systems. We can map the problem of detecting hypergraph dilations to solving a maximum matching problem in a directed star expansion.

\begin{defn}[Directed Star Expansion]
    The star expansion of $\mathcal{H}=\{\mathcal{V},\mathcal{E}_h\}$ constructs a directed bipartite graph $\mathcal{S}=\{\mathcal{V}_s,\mathcal{E}_s\}$ by introducing a new set of vertices $\mathcal{V}_s=\mathcal{V}\cup \mathcal{E}_h$ where some vertices represent hyperedges. There exists a directed edge between each vertex $v,e\in \mathcal{V}_s$ when $v\in \mathcal{V}$, $e\in \mathcal{E}_h,$ and $v\in e$. The edge is oriented $v\xrightarrow{}e$ when $v\in e^t$ and oriented $e\xrightarrow{}v$ when $v\in e^h.$
\end{defn}

We denote by $\mathcal{S}(\tA,\Bv)$ the star expansion of $\h(\tA,\Bv).$

\begin{prop}\label{prop: hypergraph dilation is star dilation}
    A hypergraph dilation in $\h(\tA,\Bv)$ exist if and only if there exist a dilation in $\mathcal{S}(\tA,\Bv)$ where the star graph dilation is directed from hyperedge representing vertices of $\mathcal{S}(\tA,\Bv)$ to vertices representing the nodes of $\h(\tA,\Bv)$
\end{prop}

From \cref{prop: hypergraph dilation is star dilation}, we can use classic algorithms to detect the presence of a dilation in $\h(\tA,\Bv)$ and verify the structural controllability of $(\tA,\Bv).$ In graphs, dilations can be detected as a maximum matching problem, and Ford-Fulkerson is a fast, classic algorithm to solve network flow problems that has been utilized to find maximum matchings and detect dilations. We can apply Ford-Fulkerson to detect dilations in our hypergraph.

\begin{algorithm}
\caption{Dilation Detector with Maximum Matching}
\label{alg:dilation detector}
\begin{algorithmic}[1]
    \REQUIRE $(\tA,\Bv)$
    \STATE Construct $\h(\tA,\Bv)=\{\V,\E\}$
    \STATE $\mathcal{S}=$star expansion of $\h(\tA,\Bv)$
    \STATE Maximum Matching = Ford-Fulkerson($\mathcal{S}$)
    \STATE Dilation = $(|$Maximum Matching$|=|\V|)$
    \RETURN Dilation
\end{algorithmic}
\end{algorithm}

The complexity of \cref{alg:dilation detector} is bounded by the complexity of Ford-Fulkerson, which is $\mathcal{O}(nm)$ where $n$ is the number of vertices and $m$ is the number of edges. In our specific case, \cref{alg:dilation detector} is $\mathcal{O}(|\V|\sum_i |e_i^h|)$ where $|e_i^h|$ is the size of the $i$th hyperhead.

\subsection{BFS for Accessibility}
Verification that all vertices are accessible can be performed in $\mathcal{O}(|\E|)$ by performing a walk beginning at control vertices that will step over each hyperedge in $\h(\tA,\Bv).$
\begin{algorithm}[h!]
\caption{Hypergraph BFS $\mathcal{O}(r|\E|)$}
\label{alg: BFS}
\begin{algorithmic}
    \REQUIRE $(\tA,\Bv)$
    \STATE Construct $\h(\tA,\Bv)=\{\V,\E\}$
    \STATE $A=\{n+1,\dots,n+m\}$
    \STATE $Q=$set of hyperedges containing hypertails formed from $A$
    \STATE $i=0$
    \WHILE{$Q$ is not empty and $i<|\E|$}
        \STATE $e=Q.pop()$
        \STATE $A=A\cup e^h$
        \STATE $N=$new accessible hyperedges from $A$
        \STATE $Q.put(N)$
        \STATE $i=i+1$
    \ENDWHILE
    \RETURN $A$
\end{algorithmic}
\end{algorithm}

\Cref{alg: BFS} returns the set of all vertices visited of a walk beginning at the control vertices. If $|A|=|\V|,$ then all vertices are accessible. It requires at most $\mathcal{O}(|\E|)$ iterations of the loop beginning on line 5 to verify this condition.

\bibliography{bibliography}

\begin{thebibliography}{18}
\providecommand{\natexlab}[1]{#1}
\providecommand{\url}[1]{\texttt{#1}}
\expandafter\ifx\csname urlstyle\endcsname\relax
  \providecommand{\doi}[1]{doi: #1}\else
  \providecommand{\doi}{doi: \begingroup \urlstyle{rm}\Url}\fi

\bibitem[Aksoy et~al.(2020)Aksoy, Joslyn, Marrero, Praggastis, and Purvine]{aksoy2020hypernetwork}
Sinan~G Aksoy, Cliff Joslyn, Carlos~Ortiz Marrero, Brenda Praggastis, and Emilie Purvine.
\newblock Hypernetwork science via high-order hypergraph walks.
\newblock \emph{EPJ Data Science}, 9\penalty0 (1):\penalty0 16, 2020.

\bibitem[Ausiello and Laura(2017)]{ausiello2017directed}
Giorgio Ausiello and Luigi Laura.
\newblock Directed hypergraphs: Introduction and fundamental algorithms—a survey.
\newblock \emph{Theoretical Computer Science}, 658:\penalty0 293--306, 2017.

\bibitem[Chen et~al.(2021)Chen, Surana, Bloch, and Rajapakse]{chen2021controllability}
Can Chen, Amit Surana, Anthony Bloch, and Indika Rajapakse.
\newblock Controllability of hypergraphs.
\newblock \emph{IEEE Transactions on Network Science and Engineering}, 8\penalty0 (2):\penalty0 1646--1657, 2021.

\bibitem[Cooper and Dutle(2012)]{cooper2012spectra}
Joshua Cooper and Aaron Dutle.
\newblock Spectra of uniform hypergraphs.
\newblock \emph{Linear Algebra and its Applications}, 436\penalty0 (9):\penalty0 3268--3292, 2012.
\newblock ISSN 0024-3795.
\newblock \doi{https://doi.org/10.1016/j.laa.2011.11.018}.
\newblock URL \url{http://www.sciencedirect.com/science/article/pii/S0024379511007610}.

\bibitem[Cowan et~al.(2012)Cowan, Chastain, Vilhena, Freudenberg, and Bergstrom]{cowan2012nodal}
Noah~J Cowan, Erick~J Chastain, Daril~A Vilhena, James~S Freudenberg, and Carl~T Bergstrom.
\newblock Nodal dynamics, not degree distributions, determine the structural controllability of complex networks.
\newblock \emph{PloS one}, 7\penalty0 (6):\penalty0 e38398, 2012.

\bibitem[Gallo et~al.(1993)Gallo, Longo, Pallottino, and Nguyen]{gallo1993directed}
Giorgio Gallo, Giustino Longo, Stefano Pallottino, and Sang Nguyen.
\newblock Directed hypergraphs and applications.
\newblock \emph{Discrete applied mathematics}, 42\penalty0 (2-3):\penalty0 177--201, 1993.

\bibitem[Gao et~al.(2014)Gao, Liu, D'souza, and Barab{\'a}si]{gao2014target}
Jianxi Gao, Yang-Yu Liu, Raissa~M D'souza, and Albert-L{\'a}szl{\'o} Barab{\'a}si.
\newblock Target control of complex networks.
\newblock \emph{Nature communications}, 5\penalty0 (1):\penalty0 5415, 2014.

\bibitem[Jurdjevic and Kupka(1985)]{jurdjevic1985polynomial}
Velimir Jurdjevic and I7996670554 Kupka.
\newblock Polynomial control systems.
\newblock \emph{Mathematische Annalen}, 272\penalty0 (3):\penalty0 361--368, 1985.

\bibitem[Kolda and Bader(2009)]{kolda2009tensor}
T.~Kolda and B.~Bader.
\newblock Tensor decompositions and applications.
\newblock \emph{SIAM Review}, 51\penalty0 (3):\penalty0 455--500, 2009.
\newblock \doi{10.1137/07070111X}.
\newblock URL \url{https://doi.org/10.1137/07070111X}.

\bibitem[Lee and Markus(1967)]{lee1967foundations}
Ernest~Bruce Lee and Lawrence Markus.
\newblock \emph{Foundations of optimal control theory}, volume~87.
\newblock Wiley New York, 1967.

\bibitem[Lin(1974)]{lin1974structural}
Ching-Tai Lin.
\newblock Structural controllability.
\newblock \emph{IEEE Transactions on Automatic Control}, 19\penalty0 (3):\penalty0 201--208, 1974.

\bibitem[Liu et~al.(2011)Liu, Slotine, and Barab{\'a}si]{liu2011controllability}
Yang-Yu Liu, Jean-Jacques Slotine, and Albert-L{\'a}szl{\'o} Barab{\'a}si.
\newblock Controllability of complex networks.
\newblock \emph{nature}, 473\penalty0 (7346):\penalty0 167--173, 2011.

\bibitem[Mayeda and Yamada(1979)]{mayeda1979strong}
Hirokazu Mayeda and Takashi Yamada.
\newblock Strong structural controllability.
\newblock \emph{SIAM Journal on Control and Optimization}, 17\penalty0 (1):\penalty0 123--138, 1979.

\bibitem[Mousavi et~al.(2018)Mousavi, Chapman, Haeri, and Mesbahi]{mousavi2018null}
Shima~Sadat Mousavi, Airlie Chapman, Mohammad Haeri, and Mehran Mesbahi.
\newblock Null space strong structural controllability via skew zero forcing sets.
\newblock In \emph{2018 European Control Conference (ECC)}, pages 1845--1850. IEEE, 2018.

\bibitem[Pickard et~al.(2023{\natexlab{a}})Pickard, Chen, Stansbury, Surana, Bloch, and Rajapakse]{pickard2023kronecker}
Joshua Pickard, Can Chen, Cooper Stansbury, Amit Surana, Anthony Bloch, and Indika Rajapakse.
\newblock Kronecker products for tensors and hypergraphs: Structure and dynamics.
\newblock \emph{arXiv preprint}, 2023{\natexlab{a}}.

\bibitem[Pickard et~al.(2023{\natexlab{b}})Pickard, Surana, Bloch, and Rajapakse]{pickard2023observability}
Joshua Pickard, Amit Surana, Anthony Bloch, and Indika Rajapakse.
\newblock Observability of hypergraphs.
\newblock \emph{arXiv preprint arXiv:2304.04883, Accepted for IEEE CDC}, 2023{\natexlab{b}}.

\bibitem[Ragnarsson and Van~Loan(2012)]{ragnarsson2012block}
Stefan Ragnarsson and Charles~F Van~Loan.
\newblock Block tensor unfoldings.
\newblock \emph{SIAM Journal on Matrix Analysis and Applications}, 33\penalty0 (1):\penalty0 149--169, 2012.

\bibitem[Shields and Pearson(1976)]{shields1976structural}
Robert Shields and J~Pearson.
\newblock Structural controllability of multiinput linear systems.
\newblock \emph{IEEE Transactions on Automatic control}, 21\penalty0 (2):\penalty0 203--212, 1976.

\end{thebibliography}

\end{document}